\documentclass[a4paper]{siamltex}
\usepackage[latin1]{inputenc}
\usepackage{amsfonts,amssymb,amsmath}
\usepackage{graphicx,pgf,verbatim}
\usepackage[normalem]{ulem}
\usepackage{enumitem}
\usepackage{hyperref}
\usepackage{url}
\usepackage{type1cm}

\def\lsim{\mathrel{\rlap{\lower4pt\hbox{\hskip1pt$\sim$}}
    \raise1pt\hbox{$<$}}}                % less than or approx. symbol

\newcommand{\hreff}[2]{#2}
\newcommand{\arxiv}[1]{\hreff{http://www.arXiv.org/abs/#1}{arXiv:#1}}

\newtheorem{remark}[theorem]{Remark}

\begin{document}
\author{Dario A. Bini
\thanks{Dipartimento di Matematica, Universit\`a di Pisa, Largo Bruno Pontecorvo 5, 56127 Pisa ({\tt bini@dm.unipi.it})}
\and 
Vanni Noferini
\thanks{Dipartimento di Matematica, Universit\`a di Pisa, Largo Bruno Pontecorvo 5, 56127 Pisa ({\tt noferini@mail.dm.unipi.it})}
\and 
Meisam Sharify
\thanks{INRIA Saclay-\^Ile-de-France and LRI, B\^at 650 Universit\'e Paris-Sud 11, 
91405 Orsay Cedex 
France ({\tt Meisam.Sharify@inria.fr})}
}

\title{Locating the eigenvalues of matrix polynomials}

\maketitle

\begin{abstract}
  Some known results for locating the roots of polynomials are
  extended to the case of matrix polynomials. In particular, a theorem
  by A.E.~Pellet [Bulletin des Sciences Math\'ematiques, (2), vol 5
  (1881), pp.393-395], some results of D.A.~Bini [Numer. Algorithms
  13:179-200, 1996] based on the Newton polygon technique, and 
  recent results of M. Akian, S. Gaubert and M. Sharify (see in particular
[LNCIS, 389, Springer p.p.291-303] and [M. Sharify, Ph.D. thesis, \'Ecole Polytechnique,
  ParisTech, 2011]).  These extensions are applied
  for determining effective initial approximations for the numerical
  computation of the eigenvalues of matrix polynomials by means of
  simultaneous iterations, like the Ehrlich-Aberth method.  Numerical
  experiments that show the computational advantage of these results
  are presented.
\bigskip

\noindent {\sl AMS classification: 15A22,15A80,15A18,47J10.}  

\noindent
\end{abstract}

{\bf Keywords:}
  Polynomial eigenvalue problems, matrix polynomials, tropical algebra,
  location of roots, Rouch\'{e} theorem, Newton's polygon.
\\

\section{Introduction}
Consider a square matrix polynomial
$A(x)=\sum_{i=0}^n A_ix^i$, where $A_i$ are $m\times m$ matrices with
complex entries and assume that $A(x)$ is regular, i.e., $a(x):=\det
A(x)$ is not identically zero. We recall that the roots of $a(x)$
coincide with the eigenvalues of the matrix polynomial $A(x)$ that is,
the complex values $\lambda$ for which there exists a nonzero vector
$v$ such that $A(\lambda)v=0$.  Computing the eigenvalues of a matrix
polynomials, known as polynomial eigenvalue problem, has recently
received much attention \cite{mehr}.

In this paper we extend to the case of matrix polynomials some known
bounds valid for the moduli of the roots of scalar polynomials like
the Pellet theorem \cite{pellet,walsh}, the Newton polygon
construction used in \cite{bini96}, applied in \cite{bf00} and
implemented in the package MPSolve
(\url{http://en.wikipedia.org/wiki/MPSolve}) for computing polynomial
roots to any guaranteed precision. We also extend a recent result by
S. Gaubert and M. Sharify~\cite{scalar2012,sha} who shed more light on
why the Newton polygon technique is so effective. Our results improve
some of the upper and lower bounds to the moduli of the eigenvalues of
a matrix polynomial given by N.~Higham and F.~Tisseur in \cite[Lemma 3.1]{ht01}.

\subsection{Motivation}
In the design of numerical algorithms for the simultaneous
approximation of the roots of a polynomial $a(x)=\sum_{i=0}^n x^i a_i$
with complex coefficients, it is crucial to have some effective
criterion to select a good set of starting values. In fact,
the performance of methods like the
Ehrlich-Aberth iteration \cite{aberth,ehrlich} or the Durand-Kerner
\cite{durand,kerner} algorithm, is strongly influenced by the choice
of the initial approximations \cite{bini96}. A standard approach, followed in
\cite{aberth,init}, is to consider $n$ values uniformly placed along a
circle of center 0 and radius $r$, say $r=1$. This choice is effective
only if the roots have moduli of the same order of magnitude.
If there are roots which differ much in modulus, then
this policy is not convenient since the number of iterations needed to
arrive at numerical convergence might become extremely large.

In \cite{bini96} a technique has been introduced, based on a theorem
by A.E. Pellet \cite{pellet,walsh},
\cite{ostro} and on the computation of the Newton
polygon, which allows one to strongly reduce the number of iterations
of the Ehrlich-Aberth method. This technique has been applied in
\cite{bf00} and implemented in the package MPSolve 
(\url{http://en.wikipedia.org/wiki/MPSolve}).
This package,
which computes guaranteed approximations to any desired precision to all the
roots of a given polynomial, takes great advantage from the Newton polygon
construction
and is one of the fastest software tools available for polynomial root-finding.

Let us introduce  the following notation to denote an annulus
 centered at the origin of the complex plane:
\begin{gather}\label{defannuli}
 \mathcal{A}(a,b):=\{x \in \mathbb{C}, \ \ a \leq |x| \leq b \},
\end{gather}
where $0<a<b$.

The theorem by A.E.~Pellet, integrated by the results of \cite{walsh}
and \cite{bini96}, states the following property.

\begin{theorem}\label{walsh}
  Given the polynomial $w(x)=\sum_{i=0}^n w_ix^i$ with $w_0,w_n\ne 0$,
  the equation
  \begin{equation}\label{uno}
    w_\kappa x^\kappa=\sum_{i=0,\, i\ne \kappa}^n |w_i|x^i
  \end{equation}
  has one real positive solution $t_0$ if $\kappa=0$, one real positive
  solution $s_n$ if $\kappa=n$, and either 0 or 2 positive real solutions
  $s_\kappa\le t_\kappa$ if $0<\kappa<n$.  Moreover, any polynomial
  $b(x)=\sum_{i=0}^n b_ix^i$ such that $|b_i|=|w_i|$, for $i=0,\ldots,
  n$, has no roots of modulus less than $t_0$, no roots of modulus
  greater than $s_n$, and no roots in the inner part of the annulus 
  $\mathcal A(s_\kappa,t_\kappa)$ if $s_\kappa$ and $t_\kappa$ exist. Furthermore, denoting
  $0=h_0<h_1<\ldots <h_p=n$ the values of $\kappa$ for which equation
  \eqref{uno} has two real positive solutions, then the number of
  roots of $b(x)$ in the closed annulus $\mathcal A(t_{h_{i-1}},s_{h_i})$
  is exactly $h_i-h_{i-1}$, for $i=1,\ldots,p$.
\end{theorem}

The bounds provided by the Pellet theorem are strict since there exist
polynomials $w(x)$ with roots of modulus $s_{h_{i-1}}$ and $t_{h_i}$.
Moreover, there exists a converse version of this theorem, given by
J.L.~Walsh \cite{walsh,ostro}. 

The Newton polygon technique, as used in \cite{bini96}, works this
way.  Given $a(x)=\sum_{i=0}^n a_ix^i$, where $a_0,a_n\ne 0$, the
upper part of the convex hull of the discrete set $\{(i,\log|a_i|),~
i=0,\ldots,n\}$ is computed. Denoting $k_i$, $i=0,...,q$ the abscissas
of the vertices such that $k_0=0<k_1<\cdots<k_q=n$, the radii
\begin{equation}\label{eq:ri}
r_i=|a_{k_{i-1}}/a_{k_{i}}|^{1/(k_i-k_{i-1})},\quad i=1,\ldots,q
\end{equation}
are formed and
$m_i=k_i-k_{i-1}$ approximations are chosen in the circle of center 0 and
radius $r_i$ for $i=1,\ldots,q$. The integer $m_i$ is called {\em multiplicity}
of the radius $r_i$.
 Observe that the sequence $r_i$ is such that $r_1\le r_2
\le\cdots\le r_q$ and that $-\log r_i$ are the slopes of the segments forming the Newton polygon.
Figure \ref{fig1} shows the Newton
polygon for the polynomial $a(x)=x^5+10^6x^4+x^3+\frac
1{100}x^2+10^3x+1$.

\begin{figure}\label{fig1}
\centering
\pgfdeclareimage[width=8cm]{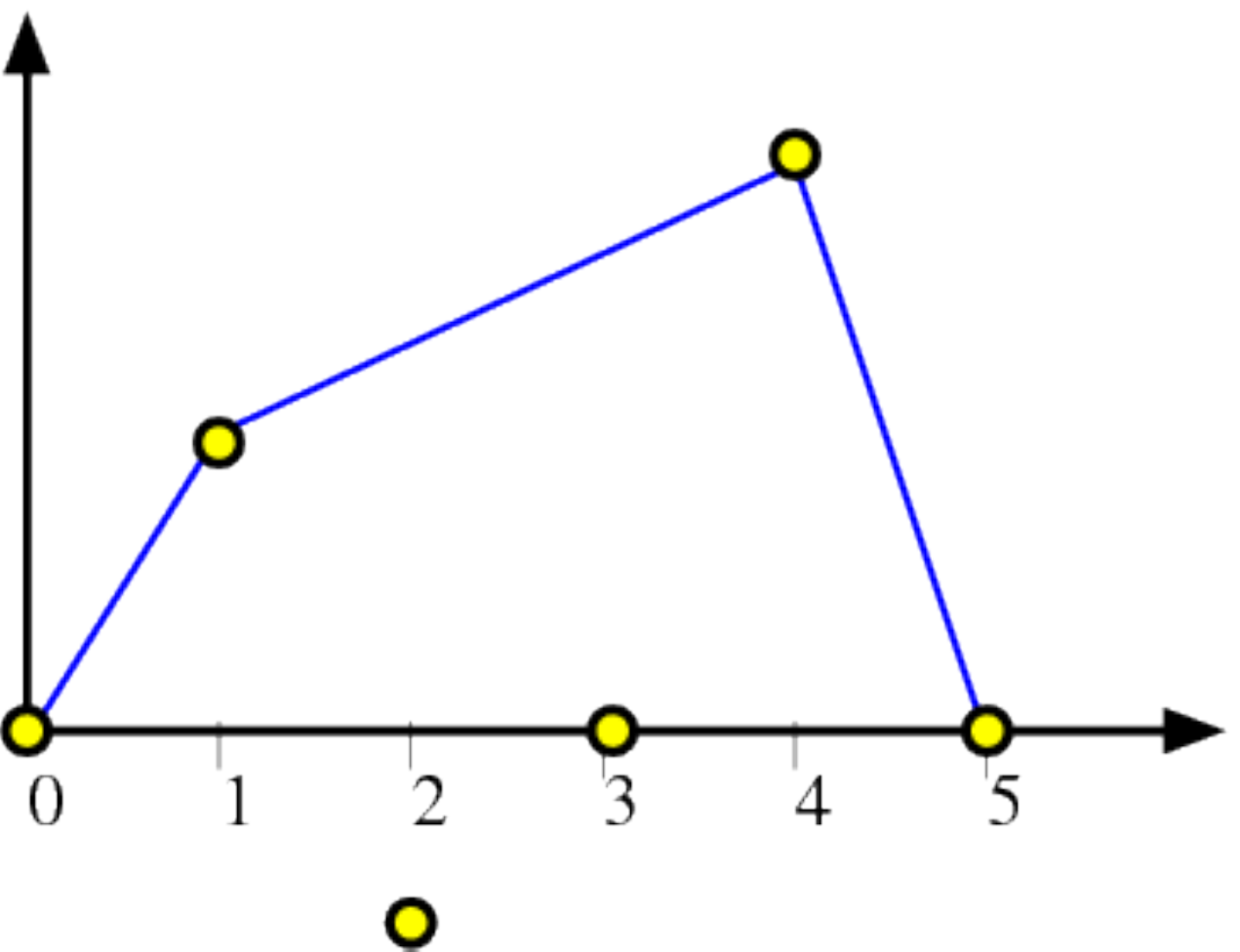}{newtpolyg}
\pgfuseimage{newtpolyg}
\caption{Newton's polygon for the polynomial
$a(x)=x^5+10^6x^4+x^3+\frac 1{100}x^2+10^3x+1$.}
\end{figure}

The effectiveness of this technique is explained in \cite{bini96}
where the following result is shown.

\begin{theorem}\label{numalgo}
Under the assumptions of Theorem \ref{walsh}, it holds $p\le q$,
\[
\{h_0,h_1,\ldots,h_p\}\subset\{k_0,k_1,\ldots,k_q\}.
\]
Moreover, $\{r_1,\ldots,r_q\}\cap ]s_{h_j},t_{h_j}[=\emptyset$ for
    $j=1,\ldots, p$.
\end{theorem} \ \smallskip

Therefore, the approximations chosen in the previously described way
lie in the union of the closed annuli $\mathcal
A(t_{h_{i-1}},s_{h_i})$, $i=1,\ldots, p$ which, according to the
Pellet theorem contain all the roots of all the polynomials having
coefficients with the same moduli of the coefficients of $a(x)$.
Moreover, the number of initial approximations chosen this way in each
annulus coincides with the number of roots that the polynomial has in
the same annulus.  The advantage of this approach is that the
computation of the Newton polygon and of the radii $r_i$ is almost
inexpensive since it requires $O(n)$ operations, while computing the
roots $s_{h_i}$ and $t_{h_i}$ is costly since it requires the solution
of several polynomial equations.  Recently, A. Melman \cite{melm} has
proposed a cheap algorithm for approximating $s_\kappa$ and $t_\kappa$.

In \cite{bini96} it is observed that any vertex $(k_i,\log|a_{k_i}|)$ 
of the Newton polygon satisfies the following property
\begin{equation}\label{eq:vert}
\begin{array}{l}
u_{k_i}\le v_{k_i},\quad v_{k_i}=u_{k_{i+1}}=r_{i+1},\quad i=0,\ldots,q-1,
\\
u_{k_i}:=\max_{j<k_i}|a_j/a_{k_i}|^{1/(k_i-j)},\quad 
v_{k_i}:=\min_{j>k_i}|a_j/a_{k_i}|^{1/(k_i-j)}.
\end{array}
\end{equation}

In certain cases, the radii $r_i$ given by the Newton polygon provide
approximations to the moduli of the roots which are better than the
ones given by the Pellet theorem. The closeness of the radii $r_i$ to
the moduli of the roots of $a(x)$, holds in particular when the radii
differ relatively much from each other, or, equivalently, when the
vertices of the Newton polygon are sufficiently sharp.

This has been recently proved by S. Gaubert and
M. Sharify~\cite{scalar2012,sha}.  
In fact, using the theory
of tropical polynomials, it turns out that the values $r_i$, for
$i=1,\ldots,q$ coincide with the so called tropical roots of $a(x)$,
and that the values $m_i=k_i-k_{i-1}$ are the multiplicities of the
tropical roots $r_i$, for $i=1,\ldots,q$. This fact is used in
\cite{scalar2012,sha} to prove the following interesting result.

\begin{theorem}\label{sha}
  If $r_{i-1}/r_{i},r_{i}/r_{i+1}<1/9$, for $i=1,\ldots,p$, then any
  polynomial $b(x)$ having coefficients with the same modulus of the
  corresponding coefficients of $a(x)$ has $k_i-k_{i-1}$ roots in the
  annulus $\mathcal{A}(r_i/3,3r_i)$.
\end{theorem}\smallskip

That is, if three consecutive radii $r_{i-1},r_i$ and $r_{i+1}$ are
sufficiently relatively far from each other, then the roots of any
polynomial having coefficients with the same modulus of the
corresponding coefficients of $a(x)$ are relatively close to the
circle of center $0$ and radius $r_i$. This explains the good
performances of the software MPSolve where only the sufficiently sharp
vertices of the Newton polygon are considered for placing initial
approximation to the roots.

An attempt to extend the Newton polygon technique to matrix
polynomials is performed in \cite{POSTA09}  by relying on tropical
algebra.  The idea consists in associating with a matrix polynomial
$A(x)=\sum_{i=0}^nA_ix^i$ the Newton polygon constructed from the
scalar polynomial $w(x)=\sum_{i=0}^n \|A_i\| x^i$.
An application of the results of \cite{POSTA09} yields a scaling
technique wich is shown to improve the backward stability of computing
the eigenvalues of $A(x)$, particularly in situations where the data
have various orders of magnitude. 
The same idea of relying on the Newton polygon constructed from $w(x)$
is used in \cite{bn} in the context of solving the polynomial
eigenvalue problem with a root-finding approach.

Moreover, it is proved
in~\cite{POSTA09} for the quadratic matrix polynomial and
in~\cite[Chapter 4]{sha} for the general case, that under assumptions
involving condition numbers, there is one group of ``large''
eigenvalues, which have a maximal order of magnitude, given by the
largest tropical root.  A similar result holds for a group of small
eigenvalues.  Recently it has been proved in~\cite{logmajor2011} that
the sequence of absolute values of the eigenvalues of $A(x)$ is
majorized by a sequence of these tropical roots, $r_i$s. This extends
to the case of matrix polynomials some bounds obtained by
Hadamard~\cite{hadamard1893}, Ostrowski and
P\'olya~\cite{Ostrowski1,Ostrowski2} for the roots of scalar
polynomials. An attempt to extend Pellet's theorem to matrix
polynomial by relying on the Gerschgorin disks has been performed by
A. Melman \cite{melm1}.

\subsection{New results} In this paper we provide extension to matrix polynomials of Theorems \ref{walsh}, \ref{numalgo}, \ref{sha} and arrive
at an effective tool for selecting initial approximations to the 
eigenvalues of a matrix polynomial. This tool, coupled with the Ehrlich-Aberth
iteration, provides a robust solver for the polynomial eigenvalue problem. A
preliminary description of an implementation of this solver is given
in \cite{bn}.

The Pellet theorem is extended by considering the equations
\[
x^\kappa=\sum_{i=0,\,i\ne \kappa}^n \|A_\kappa^{-1}A_i\| x^i,
\]
valid for all the $\kappa$ such that $\det A_\kappa\ne 0$, which have either 2 or no positive solutions
for $\kappa=1,\ldots,n-1$ and the same equations for $\kappa=0$ and $\kappa=n$ which
have one positive solution if $\det A_0\ne 0$ or $\det A_n\ne 0$,
respectively. 

The Newton polygon technique is extended by considering separately
either the set of polynomials $\sum_{i=0}^n \|A_\kappa^{-1}A_i\|x^i$ for
$\kappa$ such that $\det A_\kappa\ne 0$, or the single polynomial
$\sum_{i=0}^n\|A_i\| x^i$. The latter case is subjected to the
condition that $A_i$, $i=0,\ldots, n$ are well conditioned matrices.

Theorem \ref{sha} is extended to matrix polynomials such that
$A_i=Q_i$, $Q_i Q_i^*=\sigma_i I$ for $i=0,\ldots, n$ where the
constants 3 and 9 are replaced by slightly different values. For general
polynomials, computational evidence shows that the bounds deteriorates
when the condition number of coefficients increases.

The paper is organized as follows. In Section \ref{result} we report
the extensions of Theorems \ref{walsh}, \ref{numalgo}, \ref{sha}.
 Section \ref{proofs} contains the proofs of these extension. Finally, Section
\ref{exp} describes the results of the numerical experiments that
confirm the effectiveness of our extensions.

\section{The main extensions}\label{result}
Throughout the paper $A^*$ denotes the conjugate transpose of the matrix $A$,
$\rho(A)$ is the spectral radius and $\|A\|=\rho(A^*A)^{1/2}$ is the 2-norm.
We denote by ${\bf i}$ the imaginary unit such that ${\bf i}^2=-1$.

Define the class $\mathcal P_{m,n}$ of all the $m\times m$ matrix polynomials 
$A(x)=\sum_{i=0}^n A_ix^i$ with $A_i\in\mathbb C^{m\times m}$, satisfying the following properties:
\begin{itemize}
\item $A(x)$ is regular and has degree $n$, that is,
$a(x)=\det A(x)$ is not identically zero and $A_n\ne 0$;
\item $A_0\ne 0$.
\end{itemize}
The latter condition is no loss of generality. In fact, in this case we may just consider the polynomial $A(x)/x^\kappa$, where $\kappa$ is the smallest integer such that $A_\kappa\ne 0$.

Define also
the class $\mathcal Q_{m,n}\subset \mathcal P_{m,n}$ such that
\begin{equation}\label{ortog}
  \mathcal Q_{m,n}=\{\sum_{i=0}^n A_ix^i,\quad 
  A_i=\sigma_iQ_i,~~Q_i^* Q_i=I,~~\sigma_i\ge 0\}.
\end{equation}
The class $\mathcal Q_{m,n}$ is given by all the matrix polynomials
whose nonzero coefficients have unit spectral condition number. The
expression $\sigma Q$ provides a first step in extending the complex
number $\rho e^{{\bf i}\theta}$ to a matrix, where $\sigma$ plays the
role of $\rho$ and $Q$ of $e^{{\bf i} \theta}$.

The following result provides a first extension of the Pellet theorem.

\begin{theorem}\label{thm1} Let $A(x)\in\mathcal P_{m,n}$.
\begin{enumerate}
\item
  If $0<\kappa<n$ is such that $\det A_\kappa \ne 0$ then the equation
  \begin{equation}\label{(1)}
    x^\kappa=\sum_{i=0,\,i\ne \kappa} \|A_\kappa^{-1}A_i\| x^i      
  \end{equation}
  has either no real positive solution or two real positive solutions
  $s_\kappa\le t_\kappa$.
\item In the latter case, the polynomial
  $ a(x)=\det A(x)$ has no roots in the inner part of the annulus 
$\mathcal A(s_\kappa,t_\kappa)$, 
  while it has $m\kappa$ roots of modulus less than or equal to $s_\kappa$.
\item If $\kappa=0$ and $\det A_0\ne 0$ then \eqref{(1)} has only one real
  positive root $t_0$, moreover, the polynomial $a(x)$ has no root of
  modulus less than $t_0$.
\item If $\kappa=n$ and $\det A_n\ne 0$, then \eqref{(1)} has only one real
  positive solution $s_n$ and the polynomial $a(x)$ has no roots of
  modulus greater than $s_n$.
\end{enumerate}
\end{theorem}\smallskip

A consequence of the above result is given by the following corollary
which provides a further extension of Theorem \ref{walsh}.

\begin{corollary}\label{cor1}
Let $h_0<h_1<...<h_p$ be the values of $\kappa$ such that $\det A_\kappa\ne 0$
and there exist positive real solution(s) of \eqref{(1)}. Then
\begin{enumerate}
\item 
  $t_{h_{i-1}}\le s_{h_{i}}$, $i=1,\ldots,p$;
\item
  there are $m(h_i-h_{i-1})$ roots of $a(x)$ in the annulus
  $\mathcal A(t_{h_{i-1}} , s_{h_i})$;
\item there are no roots of the polynomial $a(x)$ in the inner part
  of the annulus $\mathcal A(s_{h_i} , t_{h_i})$, where $i=0,1,\ldots,p$ and we
  assume that $s_{h_0}=s_0:=0, t_{h_p}=t_n:=\infty$.
\end{enumerate}
\end{corollary}

Observe that in the case where $m=1$, i.e., the matrix polynomial is a
scalar polynomial, Corollary \ref{cor1} coincides with Theorem
\ref{walsh}.  Moreover, the bounds to the moduli of the eigenvalues of
$A(x)$ given in the above results are strict since there exist matrix
polynomials, say, polynomials with coefficients $A_i=\sigma_iI$, which
attain these bounds.

Theorem \ref{thm1} improves \cite[Lemma 3.1]{ht01} where the upper and lower bounds to the moduli of the eigenvalues of a matrix polynomial are given by the positive solutions of the polynomial equations
\[\begin{split}
x^n&=\sum_{i=0}^{n-1}\|A_i\|\|A_n^{-1}\|x^i;\\
1&=\sum_{i=1}^n\|A_i\|\|A_0^{-1}\|x^i.
\end{split}\]
The improvement comes from the simple observation that $\|AB\|\le \|A\|\|B\|$
for any pair of square matrices $A,B$.

If $\det A_0=0$, clearly the lower bound $t_0$ on the modulus of the
roots of $a(x)$ stated in part 3 of Theorem \ref{thm1} is missing. In
fact, there is at least one eigenvalue of $A(x)$ equal to zero. Similarly,
if $\det A_n=0$ there is no upper bound $s_n$ to the modulus of the
roots of $a(x)$ stated in part 4 of Theorem \ref{thm1}. In fact, in
this case there exist infinite eigenvalues.

Notice that in Corollary \ref{cor1} the value $h_0=0$ exists if $\det
A_0\ne 0$, and the value $h_p=n$ exists if $\det A_n\ne 0$. However,
if the remaining coefficients $A_\kappa$, for $\kappa=1,\ldots,n-1$, even though
non-singular, are very ill-conditioned, then it may happen that the set
$\{h_1,\ldots,h_{p-1}\}$ is empty so that Corollary \ref{cor1} does
not provide much information.

If $A(x)\in\mathcal Q_{m,n}$, then the following extension of
Theorem \ref{walsh} holds.

\begin{theorem}\label{th5}
  Let $A(x)\in\mathcal Q_{m,n}$ so that $A_i=\sigma_iQ_i$ and
  $Q_i^* Q_i=I$.  Let $s_{h_i},t_{h_i}$, $i=1,\ldots,p$ be the
  quantities given by Theorem \ref{walsh} applied to
  $w(x)=\sum_{i=0}^n\sigma_ix^i$. Then any matrix polynomial
  $B(x)=\sum_{i=0}^n\sigma_i S_i x^i\in\mathcal Q_{m,n}$ for
  $S_i^* S_i=I$, has
\begin{enumerate}
\item 
   $m(h_i-h_{i-1})$ eigenvalues in the annulus $\mathcal A(t_{h_{i-1}},s_{h_i})$;
\item 
   no eigenvalues with modulus in the inner part of the annulus $\mathcal A(s_{h_i},t_{h_i})$.
\end{enumerate}
\end{theorem}\smallskip

\subsection{The Newton polygon technique}
The results given in the previous subsection provide a useful tool for
selecting initial approximations to the eigenvalues of $A(x)$ to be refined
by a polynomial root-finder based on simultaneous iterations. However,
we may avoid to compute roots of polynomials and rely on the Newton
polygon construction.

In this section we provide some new results by using the Newton
polygon technique. We start by stating the following theorem.
 
\begin{theorem}\label{th6}
Let $A(x)\in\mathcal P_{m,n}$.
  If $\kappa$ is such that \eqref{(1)} has two real positive solutions
  $s_\kappa\le t_\kappa$, then $u_\kappa \le s_\kappa\le t_\kappa \le v_\kappa$ where
\[
  \begin{array}{l}
    u_\kappa:=\max_{i<\kappa}||A_\kappa^{-1}A_i||^{1/(\kappa-i)},\\[1ex]
    v_\kappa:=\min_{i>\kappa} ||A_\kappa^{-1}A_i||^{1/(\kappa-i)}
  \end{array}
\]
If $\det A_0\ne 0$ then for $\kappa=0$ \eqref{(1)} has a solution $t_0$ and
\[
  t_0 \le  v_0:=\min_{i>0}||A_0^{-1}A_i||^{-1/i}
\]
If $\det A_n\ne 0$ then for $\kappa=n$ \eqref{(1)} has a solution $s_n$ and
\[
  s_n \ge u_n:=\max_{i<n}||A_n^{-1}A_i||^{1/(n-i)}
\]
\end{theorem}\smallskip

Observe that for scalar polynomials the values $u_{k_i}$ and $v_{k_i}$ are such that
$v_{k_i}=u_{k_{i+1}}=r_i$, where $r_i$ are defined in \eqref{eq:ri} and
$k_i$ are the abscissas of the vertices of the Newton polygon.

The following result extends Theorem \ref{numalgo} to matrix polynomials.

\begin{theorem}\label{th7}
Given $A(x)\in\mathcal P_{m,n}$, let $\mathcal S=\{k_i: \quad i=0,...,q\}$ be such that $u_\kappa<v_\kappa$ 
if and only if $\kappa\in\mathcal S$, where $u_\kappa$ and $v_\kappa$ are defined in Theorem \ref{th6}. Then,
\begin{enumerate}
\item $p\le q$ and 
$\{h_0,...,h_p\} \subseteq \{ k_0,...,k_{q} \}$;
\item $
\left(\{u_1,...,u_{q}\}\cup\{v_1,\ldots,v_q\}\right) \cap [s_{k_i}, t_{k_i}] = \emptyset$;
\item $v_{k_i}\le u_{k_{i+1}}$, $i=0,\ldots,q-1$;
\item 
if $A(x)\in\mathcal Q_{m,n}$, then $v_{k_i}=u_{k_{i+1}}$ and $v_{k_i}$
coincide with the vertices of the Newton polygon of the polynomial
$w(x)=\sum_{i=0}^n\|A_i\|x^i= \sum_{i=0}^n\sigma x^i$.
\end{enumerate}
\end{theorem}\ \smallskip

Therefore, the strategy of choosing $m(k_i-k_{i-1})$ approximations
placed along the circle of center 0 and radius either $r_i=v_{k_i}$ or 
$r_i=u_{k_{i+1}}$  is
effective. In fact, these approximations lie in the union of the
closed annuli $\mathcal A_i$ of radii $t_{h_{i-1}}$ and $s_{h_i}$,
$i=1,\ldots, p$, in the complex plane which, according to the
extension of the Pellet theorem, contain all the eigenvalues of $A(x)$. The
computation of the radii $r_i$ is cheap since it is reduced to compute
the values $k_i$, $i=0,\ldots,q$ defined in Theorem \ref{th7} by
evaluating the quantities $u_{k_i}$ and $v_{k_i}$ defined in Theorem
\ref{th6}.  In the case of polynomials in $\mathcal Q_{m,n}$ this
computation is even cheaper since it is reduced to computing the
Newton polygon of the polynomial $w(x)$.

Observe that for general matrix polynomials, $u_{k_{i+1}}$ does not generally coincide with  $v_{k_i}$ nor with the values obtained by
computing the Newton polygon of $w(x)$. However, in the
practice of the computations, when the matrix coefficients
corresponding to the vertices are well conditioned, there is not much
difference between the values obtained in these different ways.

The effectiveness of this strategy of selecting starting
approximations is strengthened by the following result which
generalizes Theorem \ref{sha}.

\begin{theorem}\label{shaext}
Let $A(x)=\sum_{i=0}^n\sigma_iQ_ix^i \in\mathcal Q_{m,n}$ 
be a matrix polynomial of degree $n$
and let $r_1,\dots,r_q$ denote the radii of the Newton
polygon associated with the polynomial
$w(x)=\sum_{i=0}^n\sigma_ix^i$. Also, let 
$k_0,\dots,k_q$ be the abscissas of the vertices 
of the Newton polygon and set $m_i =k_i-k_{i-1}$.
There exist constants $f,g$ such that  $12.11<f<12.12$, $4.371<g<4.372$ and
\begin{enumerate}
\item for %the $i$th edge of the Newton polygon, 
$1<i<q$,
if $r_{i-1}/r_i,\,r_i/r_{i+1}<1/f$,
then $A(x)$ has exactly $mm_i$ eigenvalues in the 
annulus $\mathcal A(r_i/g,r_ig)$; 
\item for $i=1$, if $r_i/r_{i+1}<1/f$ then, 
$A(x)$ has exactly $mm_i$ eigenvalues in the annulus $\mathcal A(r_i/g',r_ig)$, where $g'=2+\sqrt{2}$; 
\item for $i=q$, if $r_{i-1}/r_i<1/f$ then, 
$A(x)$ has exactly $mm_i$ eigenvalues in the annulus $\mathcal A(r_i/g,r_ig')$.
\end{enumerate}
\end{theorem}\ \smallskip

Observe that in the scalar case the values of the constants $f$ and
$g$ are given by $f=9$ and $g=3$ which are slightly better.

\section{The proofs} \label{proofs}
The key tool on which the proofs of our results rely is the
generalization of Rouch\'e theorem to the case of matrix polynomials
provided in \cite{gs71}, see also \cite{rouche}. In
this statement and throughout, we use the notation $H\succ 0$ if the
Hermitian matrix $H$ is positive definite.

\begin{theorem}\label{rthm}
   Let $P(x)$ and $Q(x)$ be square matrix polynomials, and $\Gamma$ be
   a simple closed Jordan curve. If $P(x)^* P(x) - Q(x)^* Q(x)\succ 0$
   for all $x \in \Gamma$, then the polynomials $p(x):=\det(P(x))$ and
   $f(x):=\det(P(x)+Q(x))$ have the same number of roots in the open
   set bounded by $\Gamma$.
\end{theorem}

We provide the proofs of the results listed in Section
\ref{result}. Assume that $\det A_\kappa\ne 0$ and consider the matrix
polynomial $\widehat A(x)=\sum_{i=0}^nA_\kappa^{-1}A_ix^i$ which has the
same eigenvalues as $A(x)$.  We start simply by applying Theorem \ref{rthm}
to the matrix polynomials $P(x)=x^\kappa I$ and $Q(x)=\widehat A(x)-P(x)$
where $\Gamma$ is the circle of center $0$ and radius $r$.
 
The condition $P(x)^* P(x)-Q(x)^* Q(x)\succ 0$ turns into
$|x|^{2\kappa}I- Q(x)^* Q(x)\succ 0$. Moreover, since 
\[
\begin{split}
\rho(Q(x)^*Q(x))=&\|Q(x)^*Q(x)\|\le
\|Q(x)^*\|\cdot\|Q(x)\|=\|Q(x)\|^2\\
\le&(\sum_{i=0,\, i\ne \kappa}^n\|A_\kappa^{-1}A_i\|\cdot |x|^i)^2,
\end{split}
\]
 we deduce that the condition $P(x)^*P(x)-Q(x)^*Q(x)\succ 0$ is
 implied by
\begin{equation}\label{eq1}
|x|^\kappa>\sum_{i=0,\,i\ne \kappa}^n \|A_\kappa^{-1}A_i\|\cdot |x|^i.
\end{equation}

Thus, we may conclude with the following 
\begin{lemma}\label{lem1}
  If \eqref{eq1} is satisfied for $|x|=r$, then $A(x)$ has $m\kappa$ eigenvalues
  in the disk of center 0 and radius $r$.
\end{lemma}
\begin{proof}
If \eqref{eq1} is satisfied for $|x|=r$, then $P(x)P^*(x)\succ
Q(x)Q^*(x)$ for $P(x)=x^\kappa I$, $Q(x)=\widehat A(x)-P(x)$ and
$|x|=r$. Therefore, by Theorem \ref{rthm}, the matrix polynomial
$P(x)+Q(x)=\widehat A(x)$ has as many eigenvalues of modulus less than
$r$ as the matrix polynomial $P(x)=x^\kappa I$, that is $m\kappa$.
\end{proof}\ \\

We recall this known result which in \cite{bini96} is proved by
induction on $\kappa$.

\begin{lemma}\label{lem2}
  Let $w(x)=\sum_{i=0}^n w_ix^i$, $w_0,w_n> 0$, $w_i\ge 0$. The
  equation $w_\kappa x^\kappa=\sum_{i=0,\,i\ne \kappa}^nw_ix^i$ has only one real
  positive solution if $\kappa\in\{0,n\}$, and either 2 or no real positive
  solutions if $0<\kappa<n$.
\end{lemma}

Applying Lemma \ref{lem2} to the equation $x^\kappa=\sum_{i=0,\,i\ne \kappa}^n
\|A_\kappa^{-1}A_i\|\cdot x^i$, in view of Lemma \ref{lem1} one obtains
Theorem \ref{thm1}.

%% Cor 1
Now, consider the set of indices $\mathcal H=\{h_0<h_1<\ldots<h_p\}$ such that
$\det A_\kappa\ne 0$ and the equation
\[
x^\kappa=\sum_{i=0,\,i\ne \kappa}^n \|A_\kappa^{-1}A_i\|\cdot x^i.
\]
has real positive solutions for $\kappa\in\mathcal H$. Denote $s_{h_i}\le t_{h_i}$, $i=0,\ldots,p$ 
these solutions, where we have set $s_{h_0}=0$ and $t_{h_p}=\infty$. 
Observe that if $\det A_0\ne0$ then $h_0=0$ and if $\det A_n\ne 0$ then $h_p=n$.

By applying Theorem \ref{thm1} one deduces that the closed disk of
center 0 and radius $s_{h_i}$ contains exactly $mh_i$ eigenvalues of $A(x)$
for $i=1,\ldots,p$, while there are none in the inner part of the annulus
 $\mathcal A(s_{h_i},t_{h_i})$. This implies that $s_{h_{i+1}}\ge
t_{h_i}$, that is, part 1 of Corollary \ref{cor1}, and that there are
$m(h_i-h_{i-1})$ eigenvalues in the annulus
$\mathcal A(t_{h_{i-1}},s_{h_i})$, i.e., part 2 of Corollary \ref{cor1}. Part 3
follows from a direct application of Theorem \ref{thm1} so that the
proof of Corollary \ref{cor1} is complete.

%%% Thm 5
In the case where the matrix polynomial $A(x)$ belongs to $\mathcal Q_{m,n}$, 
we find that $A_\kappa^{-1}A_i=(\sigma_i/\sigma_\kappa)I$ so that condition \eqref{eq1} turns into
\[
|x|^\kappa>\sum_{i=0,\, i\ne \kappa}^n\left|\frac {\sigma_i}{\sigma_\kappa}\right|\cdot |x|^i.
\]
This way, the proof of Theorem \ref{th5} 
follows from Theorem \ref{thm1}.

\subsection{Proofs related to the Newton polygon}
Assume that $0<\kappa<n$ and that equation \eqref{(1)} has real positive solutions $s_\kappa\le t_\kappa$. Then, for any $x\in\{s_\kappa,t_\kappa\}$ one has
\[
x^\kappa= \sum_{i=0,\,i\ne \kappa}^n\|A_\kappa^{-1}A_i\| x^i\ge \|A_\kappa^{-1}A_j\| x^j
\]
for any $j\ne \kappa$. This implies that $x\ge \max_{j<\kappa}
\|A_\kappa^{-1}A_j\|^{1/(\kappa-j)}=u_\kappa$ and $x\le
\min_{j>\kappa}\|A_\kappa^{-1}A_j\|^{1/(\kappa-j)}=v_\kappa$. This proves the first part
of Theorem \ref{th6}. The cases $\kappa=0$ and $\kappa=n$ are treated similarly.

Now consider Theorem \ref{th7}. Parts 1 and 2 follow from Theorem \ref{th6}.
Concerning the inequality $v_{k_i}\le u_{k_{i+1}}$, we rely on the property
$\|H\|\ge 1/\|H^{-1}\|$ valid for any non-singular matrix $H$.
In fact, for the sake of simplicity, denote $k:=k_i$ and $h:=k_{i+1}$ so that $k<h$. Then 
\[\begin{split}
u_{h}=&\max_{j<h}\|A_h^{-1}A_j\|^{1/(h-j)}\ge\|A_h^{-1}A_k\|^{1/(h-k)}\ge\|(A_h^{-1}A_k)^{-1}\|^{-1/(h-k)}\\
& \ge \min_{j>k}\|(A_j^{-1}A_k)^{-1}\|^{-1/(j-k)}=\min_{j>k}\|A_k^{-1}A_j\|^{1/(k-j)}=v_k.
\end{split}\]
Concerning part 4, if $A(x)\in\mathcal Q_{m,n}$, then in view of Theorem \ref{th5}, the values of $k_i$ are the abscissas of the vertices of the Newton polygon for the polynomial $w(x)=\sum_{i=0}^n\sigma_i x^i$, therefore $v_{k_i}=u_{k_{i+1}}$,
and the proof of Theorem \ref{th7} is complete.

\subsection{Proof of Theorem \ref{shaext}}
Let $A(x)=\sum_{i=0}^n\sigma_iQ_i\in\mathcal Q_{m,n}$ and
$w(x)=\sum_{i=0}^n\sigma_i x^i$. 
Consider the Newton polygon corresponding to $w(x)$
%which is demonstrated in Figure~\ref{newtonpolygon}.
%Here, 
where $k_{i-1}, k_{i}$ are the abscissas of two consecutive vertices 
corresponding to the $i$th edge.
Also, let $r_1,\dots,r_q$, denote the radii corresponding to different edges of 
the Newton polygon so that $r_1<r_2<\dots<r_q$.
Along the proof we refer to the radii as the tropical roots. Also, 
for the sake of notational simplicity we set $k:=k_{i-1}$ and $h:=k_{i}$.

Let $r_i$ be the tropical root corresponding to the $i$th edge of the Newton polygon and
consider the substitution $y=r_ix$. We define the matrix polynomial $\tilde{A}(y)=\sum_{j=0}^n\tilde{A_j}y^j$ as follows:
\begin{equation}
\label{eq:scaled}
\tilde{A}(y)=(\sigma_{k}r_i^{k})^{-1}(\sum_{j=0}^nA_j(r_iy)^j)\enspace.
\end{equation}
Notice that $\lambda$ is an eigenvalue of $\tilde{A}(y)$ if and only if $r_i\lambda$ is an eigenvalue of $A(x)$.
The scaled matrix polynomial $\tilde{A}(y)$ has the following property.
\begin{lemma}[Corollary of~{\cite[Lemma~3.3.2]{sha}}]
\label{lem3}
Let $\tilde{A}(y)$ be the matrix polynomial defined in~\eqref{eq:scaled} and let 
$\tilde{\sigma_j}:=\|\tilde{A_j}\|$ for $j=0,\dots,n$. Also let 
$\epsilon:= r_{i-1}/r_{i}$, $\delta:=r_{i}/r_{i+1}$, 
be the parameters measuring the separation between $r_i$ and the previous and the next tropical roots, $r_{i-1}$ and $r_{i+1}$, respectively.
We have:
\[
\tilde{\sigma_j}\leq
\begin{cases}
\epsilon^{k-j} & \mathrm{if} \quad  j<k,\\
 1 &\mathrm{if} \quad k\leq j \leq h, \\
\delta^{j-h} & \mathrm{if} \quad j>h.
\end{cases}
\]
\end{lemma}
\begin{proof}
We only prove the first inequality %i.e. $\tilde{\sigma_j}\leq \epsilon^{k-j}$ 
since the other ones can be established by using a similar argument.
Note that $\tilde{\sigma_j}=(\sigma_{k}r_i^{k})^{-1}\sigma_jr_i^j$. 
Due to~{\cite[Lemma~3.3.2]{sha}},
$\sigma_j\leq \sigma_{k}r_{i-1}^{k-j}$ for all $0\leq j<k$. Thus,
\[
\tilde{\sigma_j}\leq 
(\sigma_{k}r_i^{k})^{-1}  \sigma_{k}r_{i-1}^{k-j}    r_i^j
= (\frac{r_{i-1}}{r_i})^{k-j}=\epsilon^{k-j}.
\]
\end{proof}
Now consider the decomposition of $\tilde{A}(y)$ as the sum of two matrix polynomials:
\[
 P(y)=\sum_{j=k}^h \tilde{A}_j y^j,\qquad
 Q(y)= \sum_{j=0}^{k-1} \tilde{A}_j y^j + \sum_{j=h+1}^n \tilde{A}_j y^j\enspace.
\]
\begin{remark}\rm
 Notice that, although for the sake of notational simplicity we do not explicitly express this dependence, 
the definitions of $\tilde{A}(y)$, $P(y)$ and $Q(y)$ depend on which edge 
of the Newton polygon we are considering.
\end{remark}
The following lemma provides the upper and lower bounds to the moduli 
of the eigenvalues of $P(y)$.
\begin{lemma}[Corollary of~{\cite[Lemma 4.1]{ht01}}]
\label{cor2}
All the nonzero eigenvalues of $P(y)$ lie in the annulus $\mathcal{A}(1/2,2)$. 
\end{lemma}
\begin{proof}
Consider the polynomial $y^{-k} P(y)$.  Due to~{\cite[Lemma 4.1]{ht01}} we have:
\[
(1+\|\tilde{A}_k^{-1}\|)^{-1}\min_{k\leq j\leq h}\|\tilde{A}_j\|^{-1/j}\leq |\lambda| \leq
(1+\|\tilde{A}_h^{-1}\|)\max_{k\leq j\leq h}\|\tilde{A}_j\|^{1/(h-k-j)}
\]
where $\lambda$ is any eigenvalue of $y^{-k} P(y)$.
The result is established by applying Lemma \ref{lem3} to the above inequalities.
\end{proof}

The idea of the proof is to 
look for the conditions on $\epsilon= r_{i-1}/r_{i}$ and $\delta=r_{i}/r_{i+1}$ 
such that $P(y)^* P(y)-Q(y)^* Q(y)\succ 0$ holds on the boundaries of two disks of center zero 
and radius $r_1<1/2$ and $r_2>2$. 
Then, by the generalized Rouch\'e theorem (Theorem~\ref{rthm}), 
$P(y)$ and $\tilde{A}(y)$ will have the same number of eigenvalues inside these disks. Using Lemma~\ref{cor2}, this implies that
$\tilde{A}(y)$ has $m(h-k)$ eigenvalues which lie in the annulus $\mathcal{A}(r_1,r_2)$; therefore,   
$A(x)$ has $m(h-k)$ eigenvalues which lie in the annulus $\mathcal{A}(r_ir_1,r_ir_2)$.
This argument is akin to the one which is used in~\cite[Chapter~3]{sha} to prove Theorem~\ref{sha}, valid for scalar
polynomials. The proof relies on the following lemmas.

\begin{lemma}\label{lem4}
Let $r:=|y|$ and $\ell:=h-k$. Also, define the diagonal matrices  
$D_i = d_i I_m$, $i=1,\dots, 3$, and the Hermitian matrices 
$H_i$, $i=1\dots 4$, as follows:
\begin{eqnarray*}
D_1(y)&:=&(\sum_{j=k}^h \|\tilde{A}_j\|^2 |y|^{2j}) I_m\enspace,
\quad 
D_2(y):=(\sum_{j=0}^{k-1} \|\tilde{A}_j\|^2 |y|^{2j}) I_m\enspace,\\
D_3(y)&:=&(\sum_{j=h+1}^n \|\tilde{A}_j\|^2 |y|^{2j}) I_m\enspace;\\
H_1(y)&:=&\sum_{k \leq j_1 < j_2 \leq h} \tilde{A}_{j_1}^* \tilde{A}_{j_2} (y^*)^{j_1} y^{j_2} +
 \left(\sum_{k \leq j_1 < j_2 \leq h} \tilde{A}_{j_1}^* \tilde{A}_{j_2} (y^*)^{j_1} y^{j_2}\right)^*,\\
H_2(y)&:=&\sum_{0 \leq j_1 < j_2 \leq k-1} \tilde{A}_{j_1}^* \tilde{A}_{j_2} (y^*)^{j_1} y^{j_2} +
 \left(\sum_{0 \leq j_1 < j_2 \leq k-1} \tilde{A}_{j_1}^* \tilde{A}_{j_2} (y^*)^{j_1} y^{j_2}\right)^*,\\
H_3(y)&:=&\sum_{h+1 \leq j_1 < j_2 \leq n} \tilde{A}_{j_1}^* \tilde{A}_{j_2} (y^*)^{j_1} y^{j_2} +
 \left(\sum_{h+1 \leq j_1 < j_2 \leq n} \tilde{A}_{j_1}^* \tilde{A}_{j_2} (y^*)^{j_1} y^{j_2}\right)^*,\\
H_4(y)&:=&\sum_{j_1=0}^{k-1}\sum_{j_2=h+1}^n \tilde{A}_{j_1}^* \tilde{A}_{j_2} (y^*)^{j_1} y^{j_2} +
\left(\sum_{j_1=0}^{k-1}\sum_{j_2=h+1}^n \tilde{A}_{j_1}^* \tilde{A}_{j_2} (y^*)^{j_1} y^{j_2}\right)^*. 
\end{eqnarray*}
Then we have:
\begin{enumerate} %\parindent
\item $[P(y)]^* P(y) - [Q(y)]^* Q(y) = D_1 (y) + H_1 (y) -D_2 (y) - D_3 (y) - H_2 (y) - H_3 (y) - H_4 (y)$;
 \item if $r>1$, $d_1 \geq r^{2 h} \frac{r^2 - 2 + r^{-2\ell}}{r^2-1}$;
\item if $r<1$, $d_1 \geq r^{2 k} \frac{1-2 r^2 + r^{2\ell+2}}{1-r^2}$;
\item $\|H_1\| \leq 2 r^{k+1} \frac{r^{2h-k+1}-r^{h+1}-r^h+r^k}{(r^2-1)(r-1)}$;
\item $d_2 + d_3 + \sum_{i=2}^4 \|H_i\| \leq \left( \epsilon \frac{\epsilon^k - r^k}{\epsilon - r} + 
\delta r^{h+1} \frac{1-(r \delta)^{n-h}}{1 -r \delta} \right)^2$.
\end{enumerate}
\end{lemma}
\begin{proof}
The first equation is easily verified by direct computation. 

The proofs of inequalities 2. and 3. have been presented in~\cite[Lemma 3.3.5]{sha}.
\begin{enumerate}[leftmargin=0cm,itemindent=.5cm,labelwidth=\itemindent,labelsep=0cm,align=left]
\item[4.]
Note that
\begin{eqnarray*}
\|H_1\| \leq 2 \sum_{k \leq j_1 < j_2 \leq h} \|\tilde{A}_{j_1}\|\|\tilde{A}_{j_2}\| r^{j_1+j_2}
&\leq& 2 \sum_{k \leq j_1 < j_2 \leq h} r^{j_1+j_2} \quad { \textit{ (using Lemma~\ref{lem3})}}\\
&=&  2 (1-r)^{-1} \sum_{j_1 = k}^{h-1}  r^{2j_1+1} (1-r^{h-j_1})\enspace.
\end{eqnarray*}
Taking the sum over $j_1$ completes the proof. 
\item[5.]
It follows from Lemma~\ref{lem3} that
\begin{eqnarray*}
&&d_2 \leq \epsilon^{2k} \sum_{j=0}^{k-1} (r/\epsilon)^{2j},\quad
 d_3 \leq \delta^{-2 h} \sum_{j=h+1}^n (\delta r)^{2j},\quad
\|H_2\|\leq 2\epsilon^{2k} \sum_{0 \leq j_1 < j_2 \leq k-1} (r/\epsilon)^{j_1+j_2}, \\
&&\|H_3\|\leq 2 \delta^{-2h} \sum_{h+1 \leq j_1 < j_2 \leq n} (\delta r)^{j_1+j_2},\quad  
\|H_4\| \leq 2 \epsilon^k \delta^{-h} \sum_{j_1=0}^{k-1} \sum_{j_2=h+1}^n (r/\epsilon)^{j_1} (r \delta)^{j_2}.
\end{eqnarray*}
Then we have
\[
d_2 + d_3 + \sum_{i=2}^4 \|H_i\| \leq \big( \epsilon^k \sum_{j=0}^{k-1} (r/\epsilon)^j
+ \delta^{-h} \sum_{j=h+1}^n (\delta r)^j\big)^2 \enspace.
\]
The proof is achieved by computing the right hand side of the above inequality.
\end{enumerate}
\end{proof}

\begin{lemma}
 \label{lem:princ1}
For any edge $i=1,\dots,q-1$ of the Newton polygon,
if $\delta=r_i/r_{i+1}<1/f$, where $12.11<f<12.12$, then $\tilde{A}(y)$ and $P(y)$
have the same number of eigenvalues in the disk centered at zero and with radius $g$, where 
$4.371<g<4.372$. 
For the last edge of the Newton polygon where $i=q$, $\tilde{A}(y)$ and $P(y)$ have the same number of eigenvalues 
in the disk centered at zero with radius $g>2+\sqrt{2}$.
\end{lemma}
\begin{proof}
Along the proof we assume that $r>1$.
Using Lemma \ref{lem4}, a sufficient condition for $[P(y)]^* P(y) - [Q(y)]^* Q(y)\succ 0$ is that
\begin{equation}\label{ustep1}
d_1 > d_2 + d_3 + \rho (-H_1 + H_2 + H_3 + H_4)\enspace.
\end{equation}
Since $\rho (-H_1 + H_2 + H_3 + H_4)\leq \sum_{i=1}^4 \|H_i\|$, ~\eqref{ustep1} is implied by
\begin{equation}\label{ustep2}
d_1-\|H_1\| > d_2 + d_3 + \sum_{i=2}^4 \|H_i\|\enspace.
\end{equation}

Observe that, by Lemma~\ref{lem4},
\[
 d_1-\|H_1\|\geq 
 r^{2 h}\big(
\frac{r^2 - 4r +2}{(r-1)^2} + \frac{2 r^{-\ell+1} -r^{-2\ell}}{(r-1)^2}\big)\enspace.
\]
Thus,~\eqref{ustep2} is deduced from the following inequality:
\begin{equation}\label{ustep3}
\frac{r^2 - 4r +2}{(r-1)^2} + \frac{2 r^{-\ell+1} -r^{-2\ell}}{(r-1)^2} > 
\left[\epsilon \ \frac{r^{-\ell} - \epsilon^kr^{-h}}{r-\epsilon} + 
\delta r \ \frac{1-(\delta r)^{n-h}}{1-\delta r}\right]^2.
\end{equation}

Assume now that $\delta r < 1$; since $\epsilon < 1$ we get
\[
\left[\epsilon \ \frac{r^{-\ell} - \epsilon^{-\ell}}{r-\epsilon} + 
\delta r \ \frac{1-(\delta r)^{n-h}}{1-\delta r}\right]^2 \leq 
\left[\frac{r^{-\ell}}{r-1} + \frac{\delta r}{1 - \delta r}\right]^2\quad \enspace.
\]
Then~\eqref{ustep3} follows from 
\begin{equation}
\frac{r^2 - 4r +2}{(r-1)^2} + \frac{2 r^{-\ell+1} -r^{-2\ell}}{(r-1)^2} >  
\left[\frac{r^{-\ell}}{r-1} + \frac{\delta r}{1 - \delta r}\right]^2 \enspace,
\label{eq:simlified}
\end{equation}
which is equivalent to
\begin{equation}
\frac{r^2 - 4r +2}{(r-1)^2} + 2r^{-2\ell}\frac{r^{\ell+1}-1+\delta r^{\ell+1}-2\delta r^{\ell+2}+\delta r}{(r-1)^2(1-\delta r)}>
%(1+\delta)r^{\ell+1}-2\delta r^{2+\ell}+\delta r-1}{(r-1)^2(1-\delta r)}>
(\frac{\delta r}{1 - \delta r})^2 \enspace.
\label{eq:step4}
\end{equation}
Now, assume that $\delta\leq \frac{1}{2r}$;
since $\ell\geq 1$, we have
\[
r^{\ell+1}-1+\delta r^{\ell+1}-2\delta r^{\ell+2}+\delta r\geq r^{\ell+1}-1-2\delta r^{\ell+2}+2\delta r=(r^{\ell+1}-1)(1-2\delta r)\geq 0\enspace.
\]
Then~\eqref{eq:step4} follows from the following inequality:
\begin{equation}\label{easyone}
\frac{r^2 - 4r +2}{(r-1)^2} > \frac{\delta^2 r^2}{(1-\delta r)^2}\enspace,
\end{equation}
which can be written in the polynomial form 
\[
 p(\delta):=a(r) \delta^2 + b(r) \delta + c(r) >0,
\]
where $a(r):=r^2(1-2r)$, $b(r):=-2r c(r)$, $c(r):=r^2-4r+2$.
Note first that $r>1$ implies $a(r)<0$; let us consider  
the discriminant of $p(\delta)$, i.e. $\Delta(r):=[b(r)]^2-4a(r)c(r)=4 r^2 (r-1)^2 c(r)$. We see that $\Delta(r)$ 
cannot be negative, otherwise $ p(\delta)<0 \ \forall r>1$. Thus, it must hold $c(r)>0$, which implies that $r>2+\sqrt{2}$.

We deduce that $p(\delta)>0$ for all the values of $\delta$ which satisfy 
\begin{equation}
\label{eq:deltaf}
 \delta<\delta_+(r):=\frac{-c(r)+(r-1) \sqrt{c(r)}}{r(2r-1)}\enspace.
\end{equation}
The graph of $\delta_+(r)$ is demonstrated in Figure~\ref{fig:delta}. 
\begin{figure}[htbp]
\centering
\includegraphics[scale=0.35]{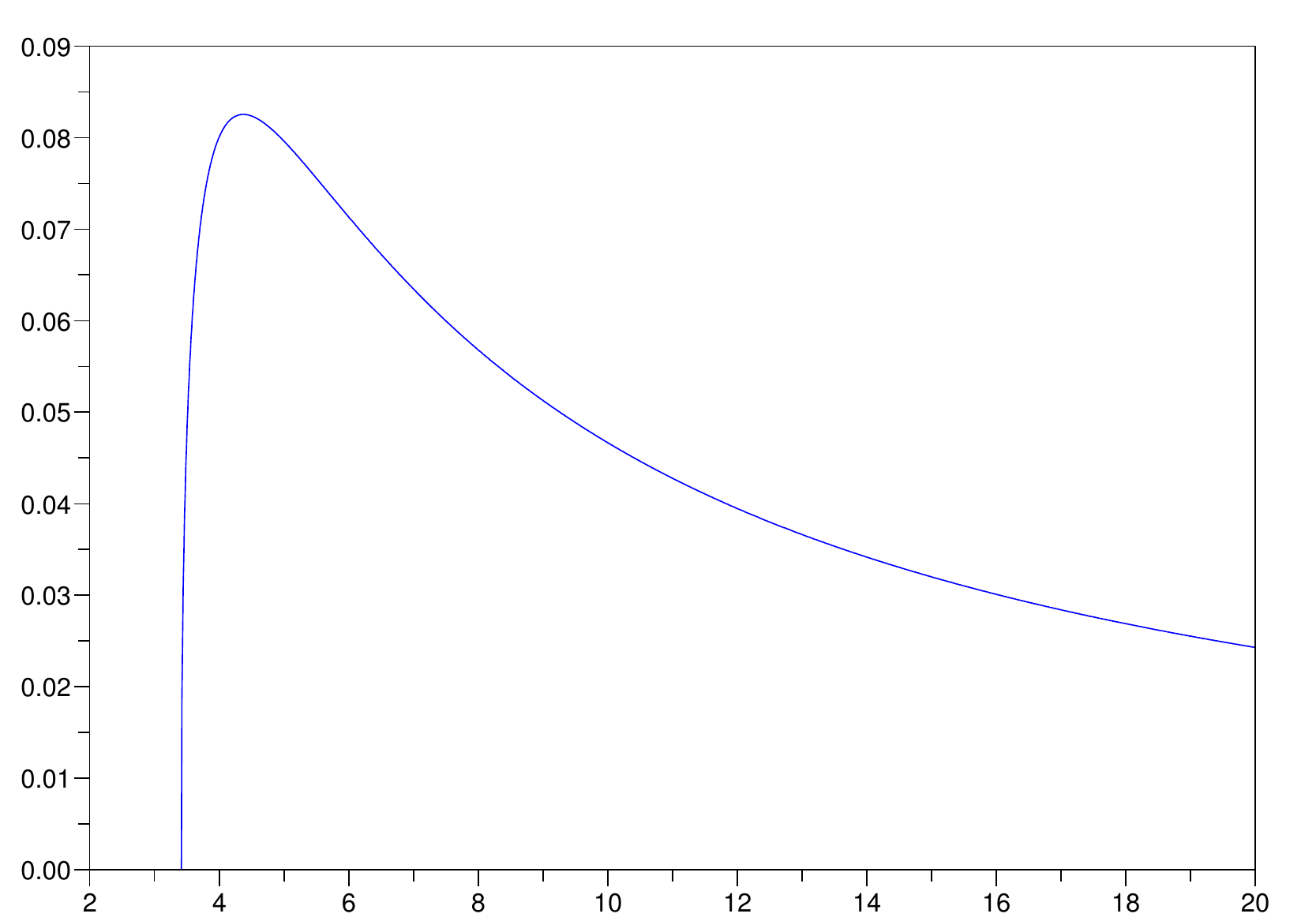}
\caption{The graph of $\delta_+(r)$ as a function of $r$.}
\label{fig:delta}
\end{figure}
The maximum value of $\delta_+(r)$
is  
$\delta_{\max} = \frac{7+3 \sqrt{3}}{2} - \sqrt{18+\frac{21 \sqrt{3}}{2}} \simeq 0.08255 \simeq \frac{1}{12.1136}$
which is obtained at $r=r_0:=\frac{3 +\sqrt{3}}{2} +\sqrt{2+\frac{7 \sqrt{3}}{6}} \simeq 4.3712$. 
So for $\delta<\delta_{\max}$
inequality~\eqref{easyone} holds, which implies that $[P(y)]^* P(y) - [Q(y)]^* Q(y)\succ 0$ on the boundary of a disk of radius
$r_0\simeq4.3712$. Then, by the Rouch\'e theorem (Theorem~\ref{rthm}), $P(y)$ and $\tilde{A}(y)=P(y)+Q(y)$ 
have the same number of zeros inside a disk of radius $r_0$. This completes the proof of the first part of the lemma.

Note now that for the last edge of the Newton polygon, since the terms $D_3,H_3,H_4$ are zero, inequality~\eqref{eq:simlified} becomes
\[
\frac{r^2 - 4r +2}{(r-1)^2} + \frac{2 r^{-\ell+1} -r^{-2\ell}}{(r-1)^2} > 
\left[\frac{r^{-\ell}}{r-1} \right]^2\enspace,
\]
or equivalently, $r^2 - 4r +2+2 r^{-\ell+1} -2r^{-2\ell}>0$
which holds for any $r>2+\sqrt{2}$.
\end{proof}

\begin{lemma}
 \label{lem:princ2}
For any edge $i=2,\dots,q$ of the Newton polygon,
if $\epsilon=r_{i-1}/r_{i}<1/f$, where $12.11<f<12.12$, then $\tilde{A}(y)$ and $P(y)$
have the same number of eigenvalues in the disk centered at zero and with radius $1/g$, where 
$4.371<g<4.372$. 
For the first edge of the Newton polygon where $i=1$, $\tilde{A}(y)$ and $P(y)$ have the same number of eigenvalues 
in the disk centered at zero, with radius $1/g$ where $g>2+\sqrt{2}$.
\end{lemma}
\begin{proof}
Along the proof we assume that $r<1$.
We follow a similar argument to the previous lemma: the formulae that we will obtain are akin to those that we had gotten for the case $r>1$. 
We will therefore give fewer details.

Starting from~\eqref{ustep2} we get:
\begin{equation}\label{ustep3bis}
\frac{2 r^2 - 4r + 1}{(1-r)^2} +  \frac{2 r^{l+1} -r^{2l+2}}{(1-r)^2} > 
\left[\epsilon\frac{1 - (\epsilon r^{-1})^{k}}{r - \epsilon} + 
\delta r^{l+1} \ \frac{1-(\delta r)^{n-h}}{1-\delta r}\right]^2,
\end{equation}
which is analogous to~\eqref{ustep3}.
Assume that $\epsilon<r$; we have
\[
\left[\epsilon\frac{1 - (\epsilon r^{-1})^{k}}{r - \epsilon} + 
\delta r^{l+1} \ \frac{1-(\delta r)^{n-h}}{1-\delta r}\right]^2 \leq 
\left[\frac{\epsilon}{r - \epsilon} + \frac{\delta r^{l+1}}{1-\delta r}
\right]^2
\leq
\left[\frac{\epsilon}{r - \epsilon} + \frac{r^{l+1}}{1-r}
\right]^2\enspace.
\]
Thus, ~\eqref{ustep3bis} is deduced from the following inequality:
\begin{equation}
\label{eq:simplifiedrs}
\frac{2 r^2 - 4r + 1}{(1-r)^2} +  \frac{2 r^{l+1} -r^{2l+2}}{(1-r)^2} > 
\left[\frac{\epsilon}{r - \epsilon} + \frac{r^{l+1}}{1-r}
\right]^2\enspace,
\end{equation}
which is equivalent to
\begin{equation}
 \label{eq:step4rb}
\frac{2 r^2 - 4r + 1}{(1-r)^2} + 2r^{l+1}\frac{r-r^{l+2}-2\epsilon+\epsilon r^{l+1}+\epsilon r}{(1-r)^2(r-\epsilon)}
>\frac{\epsilon^2}{(r - \epsilon)^2}\enspace.
\end{equation}
Now assume that $\epsilon<\frac{r}{2}$; we have
\[
r-r^{l+2}-2\epsilon+\epsilon r^{l+1}+\epsilon r\geq r-r^{l+2}+2\epsilon r^{l+1}-2\epsilon=(1-r^{\ell+1})(r-2\epsilon)>0\enspace.
\]
Thus,
~\eqref{eq:step4rb} is deduced from
\[
\frac{2 r^2 - 4r + 1}{(1-r)^2}>  
\frac{\epsilon^2}{(r - \epsilon)^2}\enspace.
\]
Setting $\rho=r^{-1}$ we get inequality~\eqref{easyone} and, following the arguments already used in the proof of Lemma~\ref{lem:princ1},
we find a maximal value equal to 
$\epsilon_{\max}= \frac{7+3 \sqrt{3}}{2} - \sqrt{18+\frac{21 \sqrt{3}}{2}} \simeq 0.08255 \simeq \frac{1}{12.1136}$, 
which is obtained at $r=r_0^{-1}=\frac{3 -\sqrt{3+2\sqrt{3}}}{2} \simeq 0.22877 \simeq \frac{1}{4.3712}$. 
For $\epsilon<\epsilon_{\max}$,
$[P(y)]^* P(y) - [Q(y)]^* Q(y)\succ 0$ on the boundary of a disk of radius
$r_0^{-1}$.

Concerning the first edge of the Newton polygon, since the terms $D_2,H_2,H_4$ are zero, ~\eqref{eq:simplifiedrs} is replaced with
\[
\frac{2 r^2 - 4r + 1}{(1-r)^2} +  \frac{2 r^{l+1} -r^{2l+2}}{(1-r)^2} > 
\left[\frac{r^{l+1}}{1-r}\right]^2\enspace,
\]
or equivalently, $2 r^2 - 4r + 1+ 2 r^{l+1} -2r^{2l+2}>0$,
which holds for any $r<1-\frac{\sqrt{2}}{2}$.
\end{proof}

\begin{proof}[Proof of Theorem~\ref{shaext}]
 By Lemmas~\ref{cor2},~\ref{lem:princ1},~\ref{lem:princ2}, $\tilde{A}(y)$ has 
$m(h-k)$ eigenvalues lying in the annulus $\mathcal{A}(1/g,g)$, where $4.371<g<4.372$.
This fact implies that $A(x)$ has $m(h-k)$ eigenvalues in the annulus $\mathcal{A}(r_i/g,gr_i)$,
where $r_i$ denotes the tropical root corresponding to the $i$th edge of Newton polygon.
\end{proof}

\begin{remark}\rm
The value of $g$ that we obtained in Lemmas~\ref{lem:princ1} and~\ref{lem:princ2} yields uniform bounds, 
independent of the exact values of $\delta$ 
and $\epsilon$. 
Yet, it is possible to get tighter bounds if either $\delta$ or $\epsilon$ are smaller than the threshold value $1/f$. 
More precisely, due to inequality~\eqref{eq:deltaf}, the
condition to be satisfied is
\[
-c(r)+(r-1) \sqrt{c(r)}-\delta(2r^2-r)>0\enspace. 
\]
Suppose that one has a given value of $\delta$, say $\delta=\delta_0<\delta_{\max}=1/f$. 
One can find the smallest $r > 2+\sqrt{2}$ such that the above inequality holds. Notice that the function $\delta_+(r)$, defined as in 
\eqref{eq:deltaf}, is increasing on the interval $(2+\sqrt{2},r_0]$. Therefore, given any $\delta_0 < 1/f$ there is a 
unique optimal radius $r=\hat{r}$ satisfying $\delta_+(\hat{r})=\delta_0$.

As an example, when $\delta=0.05$, the smallest $r$ which satisfies the above inequality is $r\simeq 3.5142$, while for
$\delta=0.01$ the smallest $r$ is $r\simeq 3.4168$ which is very close to $2+\sqrt{2}$. 
Following symmetric arguments, one 
one can show that when $\epsilon$ is much smaller than   the threshold value $\epsilon_{\max}=1/f$, the bound  on the inner radius of the annulus can be improved.
In this way, when either $\delta$ or $\epsilon$ are smaller than $1/f$, Theorem~\ref{shaext} can be modified accordingly,
providing sharper bounds for the eigenvalues of $A(x)$.
\end{remark}

\section{Numerical experiments}\label{exp}
We have created a Matlab function which implements the technique of
choosing initial approximations to the eigenvalues of a matrix
polynomial based on the results of Theorem \ref{th6}. Even though this
function is designed to deal with polynomials in the class $\mathcal
Q_{m,n}$, it can be applied to any matrix polynomial in $\mathcal
P_{m,n}$. The function works in this way: the coefficient of the
polynomial $w(x)=\sum_{i=0}^n\|A_i\|x^i$ are computed together with
the associated Newton polygon which provides the values $k_i$ and
$r_i=v_{k_i}$, $i=1,\ldots,q$.  Then $m(k_{i+1}-k_i)$ initial
approximations are uniformly placed in the circle of center 0
and radius $r_i$, for $i=1,\ldots,q$.

We have also implemented the 
Ehrlich Aberth iteration applied to the polynomial $a(x)=\det A(x)$ defined by
\begin{equation}\label{ea}
\begin{split}
&x^{(\nu+1)}_i=x^{(\nu)}_i-\frac{N(x^{(\nu)}_i)}{1-N(x^{(\nu)}_i)\sum_{j=1,\,j\ne i}^n\frac 1{x^{(\nu)}_i-x^{(\nu)}_j}},\quad i=1,\ldots,mn,\quad \nu=0,1,\ldots,\\
&N(x)=a(x)/a'(x)
\end{split}
\end{equation}
starting from the initial approximations $x^{(0)}_i$, $i=1,\ldots,mn$,
where the Newton correction $N(x)$ is computed with the formula $N(x)=
1/\hbox{trace}(A(x)^{-1}A'(x))$.  In this implementation, the
iteration \eqref{ea} is applied only to the components $x^{(\nu)}_i$
for which the numerical convergence has not occurred yet.  We
say that $x^{(\nu)}_i$ is numerically converged if either the Newton
correction is relatively small, i.e., if $|N(x^{(\nu)}_i)|\le\epsilon
|x^{(\nu)}_i|$ for a suitable $\epsilon>0$ close to the machine
precision, or if the reciprocal of the condition number of
$A(x^{(\nu)}_i)$ is smaller than a given $\delta>0$ close to the
machine precision.  The execution is halted if either $\nu=5000$ or if
all the components $x^{(\nu)}_i$ have arrived at numerical
convergence.

Observe that each component $x^{(\nu)}_i$ may converge with a number
of iterations depending on $i$ so that each simultaneous iteration in
\eqref{ea} does not have the same computational cost. In fact, while
in the initial steps all the components in \eqref{ea} must be updated,
in the subsequent steps, when most of the components have arrived at
convergence, only a few components $x^{(\nu)}_i$ must be updated. For
this reason, it is not fair to compare performances by relying only on
the maximum value reached by the parameter $\nu$ which counts the
number of simultaneous iterations.

Therefore, in our experiments besides counting the maximum number of
simultaneous iterations {\tt simul\_it}, that is, the maximum value
reached by $\nu$, we have taken into account the average number of
iteration per component {\tt aver\_it}, given by ${\tt aver\_it}=\frac
1{mn}\sum_{i=1}^{mn} \nu_i$, where $\nu_i$ is the number of steps
needed for convergence of the $i$th component $x^{(\nu)}_i$. The
quantity {\tt aver\_it} is more meaningful and represents the number
of simultaneous iterations that one would obtain if all the components
require the same number of iteration to arrive at convergence.
 The value of {\tt simul\_it} might be meaningful in a parallel
environment where each processor can execute the iteration on a given component
$x_i$.

We have computed the values of {\tt simul\_it} and {\tt aver\_it} obtained 
by applying the Ehrlich-Aberth iteration starting with initial approximations $x^{(0)}_i$ uniformly placed along the unit circle and with initial approximations placed according to our criterion.

The first set of experiments concerns matrix polynomials in the class
$\mathcal Q_{m,n}$, i.e., polynomials with coefficients $A_i=\sigma_i
Q_i$ with $Q_i^*Q_i=I$ and $\sigma_i\ge 0$. The matrices $Q_i$ have
been chosen as the orthogonal factors of the QR factorization of 
randomly generated
matrices. Concerning the scalars $\sigma_i$ we have set
\[\begin{split}
\sigma=&[{\tt 1,~ 3.e5,~ 3.e10,~ 1.e15,~ 0,~ 0,~ 
0,~ 0,~ 0,~ 1.e40,~ 0,~  0,~ 0,~ 1]}
\end{split}
\]
so that $A(x)$ is a polynomial of degree 13 with eigenvalues of unbalanced
moduli.  We have chosen different values for $m$, more precisely,
$m=5,10,20,40$. Table \ref{tab1} reports the number of iterations. It
is interesting to point out that the reduction factor in the number of
average and simultaneous iterations is quite large and grows almost 
linearly with the size $m$.

\begin{table}
\begin{center}
\begin{tabular}{|l|ll|ll|}
\hline
& \multicolumn{2}{c|} {Newton polygon}
& \multicolumn{2}{c|} {Unit circle}
\\
$m$ & {\tt simul\_it} & {\tt aver\_it} 
  & {\tt simul\_it} & {\tt aver\_it} 
\\
\hline
5 &\tt 8   &\tt 5.4       &
   \tt 243  &\tt 191  
\\
\hline
10 &\tt 9   &\tt 5.5       &
   \tt 444  &\tt 375 
\\
\hline
20 &\tt 11  &\tt  5.6     &
   \tt 855  &\tt  738
\\
\hline
40 &\tt 13  &\tt 6.1      &
   \tt 1594  &\tt 1466  
\\
\hline
\end{tabular}\caption{Number of simultaneous iterations and average 
iterations needed by the Ehrlich-Aberth iteration by choosing initial
approximations on the unit circle or by using the Newton polygon
technique. Polynomial with orthogonal coefficients scaled with
constants $(\sigma_i)=[{\tt 1,3.e5, 3.e10, 1.e15, 0, 0, 0, 0, 0,
    1.e40, 0, 0, 0, 1]}$.}\label{tab1}
\end{center}
\end{table}

We have also applied this technique to polynomials with randomly
generated coefficients, which are not generally orthogonal, scaled by
the same factors $\sigma_i$. The results are reported in Table
\ref{tab2}. We may observe that the behavior is almost the same: the
proposed strategy for choosing initial approximations still leads to a
substantial decrease of the number of iterations.  It must be said
that in the test performed, the condition numbers of the block
coefficients corresponding to the vertices of the Newton polygon is
not very large, the largest value encountered was around {\tt 5.0e3}.

\begin{table}
\begin{center}
\begin{tabular}{|l|ll|ll|}
\hline
& \multicolumn{2}{c|} {Newton polygon}
& \multicolumn{2}{c|} {Unit circle}
\\
$m$ & {\tt simul\_it} & {\tt aver\_it} 
  & {\tt simul\_it} & {\tt aver\_it} 
\\
\hline
5 &\tt 9   &\tt 6.8       &
   \tt 240  &\tt 190  
\\
\hline
10 &\tt 13   &\tt 7.7       &
   \tt 457  &\tt 372 
\\
\hline
20 &\tt 16  &\tt  9     &
   \tt 851  &\tt  732
\\
\hline
40 &\tt 16  &\tt 10.4      &
   \tt 1597  &\tt 1457  
\\
\hline
\end{tabular}\caption{Number of simultaneous iterations and average 
iterations needed by the Ehrlich-Aberth iteration by choosing initial
approximations on the unit circle or by using the Newton polygon
technique. Polynomial with random coefficients scaled by constants
$(\sigma_i)=[{\tt 1,3.e5, 3.e10, 1.e15, 0, 0, 0, 0, 0, 1.e40, 0, 0, 0,
    1]}$.}\label{tab2}
\end{center}
\end{table}

\section{Conclusions}
Some known results valid for estimating the moduli of the roots of a
polynomial have been extended to the case of matrix polynomials. These
results have been applied to design a polynomial eigenvalue solver
based on the Ehrlich-Aberth iteration. We have shown the effectiveness
of our approach by means of numerical experiments. We plan to exploit
these results in order to arrive at the implementation of a
multiprecision matrix polynomial root-finder analogous to MPSolve
\cite{bf00}.

\section*{Acknowledgment}The first author wishes to thank Aaron Melman for pointing out reference \cite{gs71} and for illuminating discussions.

%\section*{Bibliography}

\end{document}